\documentclass[twoside,12pt]{article}
\usepackage{amsthm,amsmath,amssymb,amscd,enumerate,epsfig, amsfonts, mathrsfs, graphicx, MnSymbol,calrsfs, hyperref, tikz, caption}
\usepackage{amsfonts}
\usepackage{pst-all,epsfig,fancybox}
\usepackage{graphicx}
\usepackage{fancyhdr}
\begin{document}
\textheight=570pt 
\begin{center} {\Large\bf  Dominion of Some Graphs}
\vskip 1cm {\bf  Julian A.D. Allagan$^{1}$, Benkam Bobga$^{2}$  }\\
\
\\$^{1}$Department of Mathematics, Computer Science and Engineering Technology\\
Elizabeth City State University \\
	Elizabeth City, North Carolina, U.S.A 
\\
\
\\$^{2}$Department of Mathematics\\
University of North Georgia\\
Gainesville, Georgia, U.S.A
\\
\
\\email: adallagan@ecsu.edu, benkam.bobga@ung.edu \\
\end{center}
\begin{center}
(January 28, 2021)
\end{center}
\begin{abstract}
			Given a graph $G = (V,E)$, a subset $S \subseteq V$ is a dominating set if every vertex in $V\setminus S$ is adjacent to some vertex in $S$. The dominating set with the least cardinality, $\gamma$, is called a $\gamma$-set which is commonly known as a minimum dominating set. The dominion of a graph $G$, denoted by $\zeta(G)$, is the number of its $\gamma$-sets. Some relations between these two seemingly distinct parameters are established. In particular, we present the dominions of paths, some cycles and the join of any two graphs.   
\end{abstract}
\begingroup
\newtheorem{conjecture}{Conjecture}[section]
\newtheorem{theorem}{Theorem}[section]
\newtheorem{lemma}[theorem]{Lemma}
\newtheorem{proposition}[theorem]{Proposition}
\newtheorem{corollary}[theorem]{Corollary}
\newtheorem{definition}[theorem]{Definition}
\newtheorem{remark}[theorem]{Remark}
\newtheorem{example}[theorem]{Example}
\newtheorem{claim}[theorem]{Claim}
\endgroup
\def\E#1{\left\langle #1 \right\rangle}
\def\diag{\mathop{\fam 0\relax diag}\nolimits}
\newcommand{\ds}{\displaystyle}
\section{Introduction}

A {\it dominating set} for a graph $G = (V,E)$ is a subset $S \subseteq V$ such that every vertex $v\in V$ is either in $S$ or has a neighbor $u\in S$. The vertex $u \in S$ is said to \textit{cover}  the vertex $v \in V$ if either $u=v$ or $uv\in E$. A dominating set $S$ is a \textit{minimal dominating set} if no proper subset $S'\subset S$ is a dominating set. The \textit{domination number} $\gamma(G)$ of a graph $G$ is the minimum cardinality among all dominating sets of $G$. Clearly, a graph $G$ can have multiple minimum dominating sets. For simplicity, we refer to those sets as \textbf{$\gamma$-sets}. 

Dominating sets, in particular $\gamma$-sets, have been extensively researched. We recommend to the reader Haynes et al.'s book \cite{Hay2} and several other research works \cite{FB, BR, EC, GD, MH, FH, Hay1, ML} on variants and applications of dominating sets.
In this literature we found various versions of $\gamma$-set such as the perfect dominating set, the double dominating set, the strong dominating set, and the restrained dominating set. In general, each such type of $\gamma$-set has some specified properties. Here, we list five of the most commonly studied $\gamma$-sets along with their properties. Later in Example \ref{Ex}, we illustrate each such $\gamma$-set.

\begin{enumerate}
	\item \textit{Perfect $\gamma$-set} requires that no two vertices in a $\gamma$-set cover the same vertex (outside the $\gamma$-set). \item \textit{Connected $\gamma$-set} requires that the graph induced by the $\gamma$-set be connected.
	\item \textit{Total $\gamma$-set} forbids isolated vertices in the subgraph induced by the $\gamma$-set.
	\item \textit{Independent $\gamma$-set} requires all vertices in a $\gamma$-set to be isolated vertices in the subgraph they induce.
	\item \textit{clique $\gamma$ set} requires that the vertices in a $\gamma$-set induce a clique.
\end{enumerate}

Because a given graph may have multiple $\gamma$-sets, this article addresses a natural but fundamental question:\textit{ How many  $\gamma$-sets does a given graph have?} To answer this question, we introduce the notion of \textbf{dominion}.

The \textit{dominion (number)} of  a graph $G=(V,E)$, denoted by $\zeta(G)$, is the number of its $\gamma$-sets. In other words, 

$\zeta(G):=|\{S: S \text{ is a dominating set in} \ G \ \text{and} \ |S|=\gamma\}|$. For instance, given a  complete graph $G=K_n$, it is obvious that each $v\in V(G)$ covers the remaining $n-1$ vertices of $G$. Therefore $\zeta(G)=n$ while $\gamma(G)=1$. Also, for a star graph $G=K_{1,n}$ of order $n+1\ge 3$, it is easy to see that $\gamma(G)=1=\zeta(G)$ since only the central vertex can cover all of the remaining $n\ge 2$ vertices. 

\begin{example}{\label{Ex}}
\end{example}
Here, we illustrate the $\gamma$-sets of a unicyclic graph $G$ (see Figure \ref{fig1}). 
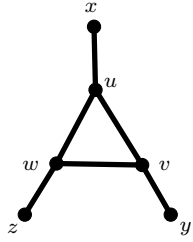
\begin{figure}[ht]
	\centering
	\begin{tikzpicture}[scale=1]
	\clip(0,-1) rectangle (4,4);
	\draw [line width=2pt] (2.62,1)-- (2,2);
	\draw [line width=2pt] (2,2)-- (1.98,2.84);
	\draw [line width=2pt] (2.62,1)-- (3,0.34);
	\draw [line width=2pt] (2,2)-- (1.48,1.02);
	\draw [line width=2pt] (1.48,1.02)-- (2.62,1);
	\draw [line width=2pt] (1.48,1.02)-- (1.06,0.34);
	\begin{scriptsize}
	\draw [fill=black] (2,2) circle (2.5pt);
	\draw[color=black] (2.2,2.1) node {$u$};
	\draw [fill=black] (2.62,1) circle (2.5pt);
	\draw[color=black] (2.92,1) node {$v$};
	\draw [fill=black] (1.98,2.84) circle (2.5pt);
	\draw[color=black] (1.94,3.1) node {$x$};
	\draw [fill=black] (3,0.34) circle (2.5pt);
	\draw[color=black] (3.2,0.15) node {$y$};
	\draw [fill=black] (1.48,1.02) circle (2.5pt);
	\draw[color=black] (1.14,1) node {$w$};
	\draw [fill=black] (1.06,0.34) circle (2.5pt);
	\draw[color=black] (0.9,0.15) node {$z$};
	\end{scriptsize}
	\end{tikzpicture}
	\caption{A sun graph $G$ with $\gamma(G)=3$ and $\zeta(G)=8$} \label{fig1}
	
\end{figure}

Consider the unicyclic graph $G$ as shown in Figure \ref{fig1}. It is easy to see that $\gamma(G)=3$. We list all the $\gamma$-sets with properties which are among the five commonly studied, as  previously mentioned.

\begin{itemize}

	\item $\{u,v,w\}$ is a perfect $\gamma$-set, a connected $\gamma$-set, a total $\gamma$-set, and a clique $\gamma$-set.
	\item $\{x,y,z\}$ is a perfect $\gamma$-set, and an independent $\gamma$-set.
	 \item $\{u,y,z\}$, $\{v,x,z\}$, and $\{w,x,y\}$ are independent $\gamma$-sets but not perfect, connected, clique or total. 
	\item $\{u,v,z\}$, $\{u,w,y\}$, and $\{v,w,x\}$ are  $\gamma$-sets which have none of the $5$ special properties that we previously listed.
\end{itemize}
We have $\ds \zeta(G)=\sum_{i=0}^3{3\choose i}=8$, the number of $\gamma$-sets of $G$.
We present a generalization of the previous example for such unicyclic graphs which are commonly known as \textit{sun graphs}.

\begin{proposition}
   If $G$ is a sun graph on $2n\ge 6$ vertices, then $\ds \zeta(G)=\sum_{i=0}^n{n\choose i}=2^n. $  
\end{proposition}
    
\begin{proof}
For each vertex on the cycle, and for each dominating set $S\subseteq V(G)$, either the vertex or its leaf neighbor must be in $S$. Therefore, we can form all $\gamma$-sets $S$ in $G$ from choosing subsets of the set of $n$ leafs, and the cycle neighbors of unchosen leafs. This gives  $\ds \sum_{i=0}^n{n\choose i}$.  Conversely, every $S$ assembled by choosing either the vertex on the cycle or the leaf, for each leaf, is a $\gamma$-set for $G$. This gives $\zeta(G)=2^n$.
\end{proof}


\section{Dominion of Some Graphs}

Suppose $P_n:=v_1-v_2-\ldots-v_{n-1}-v_n$, denotes a path on $n\ge 2$ vertices. We establish each of the following claims:
\begin{claim}{\label{c1}}
	There is no $\gamma$-set that includes both vertices $v_{n-1},v_n$. 
\end{claim}
\begin{proof}
	Suppose $n=3k+1$. It is easy to see that $\gamma(P_{3k+1})=k+1$ and without loss, let's assume that there is a $\gamma$-set $S$, with $|S|=k+1$, that includes both  $v_{3k},v_{3k+1}$. Since $v_{3k-1}$ is covered by $v_{3k}$ (and vice versa), this leaves exactly $3k-2$ vertices of $P_{3k+1}$ that are to be covered using the remaining $k-1$ vertices of $S$. Because each vertex of $S$ covers at most $2$ other vertices of $P_{3k+1}$, $k-1$ vertices can cover at most $3(k-1)$ vertices of the path, and $3(k-1)<3k-2$. A similar argument addresses the case when $n=3k+2$ and the case when $n=3k$ follows from a pigeonhole principle argument, since $\gamma(P_{3k})=k$.
\end{proof}

\begin{claim}{\label{c2}}
	There is no $\gamma$-set that excludes both vertices $v_{n-1},v_n$. 
\end{claim}
\begin{proof}
	Suppose $S$ is a $\gamma$-set of $P_n$ that excludes both consecutive vertices $v_{n-1},v_n$. It follows that $v_n\in V(P_n)$ is left uncovered. Hence, $S$ is not a $\gamma$-set of $P_n$.
\end{proof}

\begin{claim}{\label{c3}}
	Every $\gamma$-set includes either $v_{n-1}$ or $v_n$ but not both. 
\end{claim}

\begin{claim}{\label{c4}}
	When $n=3k, k\ge 1$, $\zeta(P_n)=1$ and the unique $\gamma$-set in $P_n$ contains $v_{n-1}$. Moreover, when $n=3k+2, k\ge 1$, there is only one $\gamma$-set in $P_n$ containing $v_{n}$. 
\end{claim}
\begin{proof}
	When $n=3k, k\ge 1$,  we have $\zeta(P_n)=k$ and the conclusion follows by the pigeon-hold argument alluded in Claim \ref{c1}. Each of the sets of $3$ consecutive vertices $\{v_{3r+1},v_{3r+2},v_{3r+3}\}$, $r=0,\ldots,k-1$, must contain at least one vertex of any $\gamma$-set $S$, and since $\gamma=k$, each of those sets must contain exactly one element of $S$. Since $S$ is dominating, the element of $S$ in $\{v_1,v_2,v_3\}$ must be $v_1$ or $v_2$. If $v_1$ then the ``next'' element of $S$ must be $v_4$, which forces $S=\{v_1,v_4,\ldots,v_{3k-2}\}$, leaving $v_{3k}$ uncovered. Therefore, the first element of $S$ must be $v_2$, which forces $S=\{v_2,v_5,\ldots,v_{3k-1}, v_{3k+2}\}$.
	
Suppose that $n=3k+2$, and let $S$ be a $\gamma$-set in $V(P_n)$. Then $|S|=k+1$. Suppose that $v_n\in S$. Then $v_{n-1}\notin S$, so $S-\{v_n\}$ must be a dominating set in $P_n-\{v_n,v_{n-1}\}=P_{3k}$. Since $|S-\{v_n\}|=k=\gamma(P_{3k})$ it follows from the previous argument that $S=\{v_2,v_5,\ldots, v_{3k-1}, v_{3k+2}\}$. 
	
\end{proof}
\begin{theorem}{\label{T1}}
	If $G$ is a path on $n$ vertices, then its dominion is \\
	
	$\zeta(G)=\begin{cases} 
	1, & \text{if} \ n \equiv 0\pmod3 \hspace{.5in} (i)\\
	\frac{(n+2)(n+11)}{18}-1, &  \text{if} \ n \equiv 1\pmod3 \hspace{.5in} (ii)\\
	{\lceil \frac{n}{3} \rceil}+1, & \text{if} \ n \equiv 2\pmod3. \hspace{.5in} (iii)
	\end{cases}$
\end{theorem}

\begin{proof}
	Part (i). This follows directly from Claim \ref{c4}. 
	\newline
	
	Part (iii).
   	Suppose that $n=3k+2$, $k\ge 1$, and $G=P_n$. We have $\gamma(G)=k+1$ and every $\gamma$-set $S\subseteq V(G)$ contains exactly one of $v_n$, $v_{n-1}$. From Claim \ref{c4} we know that exactly one such $S$ contains $v_n$, so our task is to show that exactly ${\lceil \frac{n}{3} \rceil}=k+1$ such $S$ contain $v_{n-1}$.
	Suppose that $S$ is a $\gamma$-set for $G$ containing $v_{n-1}$. By the reasoning in the proof of Claim \ref{c4}, each $3$-set $\ds T_r=\{v_{3r+1},v_{3r+2},v_{3r+3}\}$, $r=0,\ldots,k-1$, contains exactly one vertex of $S$. 
	\indent If $0\le r< k-1$ and $\ds S\cap T_r=\{v_{3r+1}\}$ then, because $S$ is dominating in $G$, $S\cap T_{r+1}=\{v_{3r+4}\}$; consequently, if $S\cap T_r=\{v_{3r+1}\}$ then $S\cap T_s=\{v_{3s+1}\}$, $r\le s \le k-1$. 
	
	\indent On the other hand, if $S\cap T_r=\{v_{3r+2}\}$ then $S\cap T_{r+1}$ is either $\{v_{3r+4}\}$ or $\{v_{3r+5}\}$. Since $S\cap T_{0}$ is either $\{v_1\}$ or $\{v_2\}$ (since $v_3$ does not cover $v_1$), it follows that the $\gamma$-set for $G=P_n$, $n=3k+2$, containing $v_{n-1}$ are the following:
	\begin{align}
	    S_0&=\{v_1,v_4,\ldots,v_{n-4},v_{n-1}\},\nonumber \\
	     S_1&=\{v_2,v_4,\ldots,v_{n-4},v_{n-1}\},\nonumber \\
	     \vdots \nonumber \\
	     S_r&=\{v_2,v_5,\ldots,v_{3r-1},v_{3r+1},\ldots,  v_{n-4},v_{n-1}\},\nonumber \\
	     \vdots \nonumber \\
	     S_k&=\{v_2,v_5,\ldots,v_{n-3},v_{n-1}\}, \nonumber
	\end{align}
and there are exactly $k+1$ of these sets.

	Part(ii)
	We recursively obtain the $\zeta$ values of $G_k=P_{3k+1}:=v_1-v_2-\ldots-v_{3k}-v_{3k+1}$ as we increase the value of $k\ge 1$. For simplicity, let $\zeta(G_k)=\zeta_k$, $k\ge 1$.
	
	When $k=1$, it is easy to verify that the $\gamma$-sets of $G_1$ are $\{v_1,v_3\}$, $\{v_1,v_4\}$,$\{v_2,v_3\}$, $\{v_2,v_4\}$. Hence $\zeta_1=4$.
	
	When $k=2$, it is clear that $\gamma(G_2)=3$. Since $v_6$ covers $v_5$ and $v_7$, we add $v_6$ to each of the $\gamma$-sets of $G_1$ to get $\{v_1,v_3, v_6\}$, $\{v_1,v_4, v_6\}$, $\{v_2,v_3, v_6\}$, $\{v_2,v_4, v_6\}$. Further, we add $v_7$ to each $\gamma$-set of $G_1$ that includes $v_4$ since $v_4$ covers $v_5$, and $v_7$ covers $v_6$. This produces two additional $\gamma$-sets for $G_2$, namely $\{v_1,v_4, v_7\}, \{v_2,v_4, v_7\}$. Finally, with the new and unique set $\{v_2, v_5\}$ which covers all other vertices except $v_7$, we generate two more $\gamma$-sets of $G_2$ using $v_6$ for one and $v_7$ for the other. It is clear that $v_6$ covers $v_7$ (and vice-versa), giving  $\{v_2,v_5, v_6\}, \{v_2, v_5, v_7\}$. Hence $\zeta_2=8$.
	
	When $k=3$, it is also clear that $\gamma(G_3)=4$. We use a very similar construction as in the case when $k=2$. Since $v_9$ covers $v_8$ and $v_{10}$, we add $v_9$ to each of the $8$ $\gamma$-sets of $G_2$ to get 
	$\{v_1,v_3, v_6, v_9\}$, $\{v_1,v_4, v_6, v_9\}$, $\{v_2,v_3, v_6, v_9\}$, $\{v_2,v_4, v_6, v_9\}$, $\{v_1,v_4, v_7, v_9\}, \{v_2,v_4, v_7, v_9\}$, $\{v_2,v_5, v_6, v_9\}, \{v_2, v_5, v_7, v_9\}$
	. Further, we add $v_{10}$ to each $\gamma$-set of $G_2$ that includes $v_7$ since $v_7$ covers $v_8$, and $v_{10}$ covers $v_{9}$. This produces three additional $\gamma$-sets for $G_3$, namely $\{v_1,v_4, v_7, v_{10}\}, \{v_2,v_4, v_7, v_{10}\}, \{v_2, v_5, v_7, v_{10}\}$. Finally, with the unique set $\{v_2, v_5, v_8\}$ which covers all other vertices except $v_{10}$, we generate two more $\gamma$-sets of $G_3$ using $v_9$ for one and $v_{10}$ for the other. Thus, $v_9$ covers $v_{10}$ (and vice-versa) giving  $\{v_2,v_5, v_8, v_9\}, \{v_2, v_5, v_8, v_{10}\}$. Hence $\zeta_3=13$. Now, we generalize the previous constructions for all $k\ge 2$.
	
	It is obvious that $\gamma(G_k)=k+1$. So, for $k\ge 2$, we repeatedly form the $\gamma$-sets of $G_k$ by extending those of $G_{k-1}$ as follows:
	
	\begin{enumerate}
		\item[(a)] Since $v_{3k}$ covers both $v_{3k-1}$ and $v_{3k+1}$, we add $v_{3k}$ to each $\gamma$-set of $G_{k-1}$.
		\item[(b)] We add $v_{3k+1}$ to each $\gamma$-set of $G_{k-1}$ that includes $v_{3k-2}$ since $v_{3k-2}$ covers $v_{3k-1}$, and $v_{3k+1}$ covers $v_{3k}$.
		\item[(c)] Given the unique set $\{v_2,v_5,\ldots,v_{3k-1}\}$ which covers all other vertices of $G_k$ except $v_{3k+1}$, we generate two more $\gamma$-sets of $G_k$ using $v_{3k}$ for one and $v_{3k+1}$ for the other. Clearly $v_{3k}$ covers $v_{3k+1}$ (and vice-versa) and we have $\{v_2,v_5,\ldots,v_{3k-1}, v_{3k}\}$ and $\{v_2,v_5,\ldots,v_{3k-1}, v_{3k+1}\}$. 
	\end{enumerate}
	From steps (a)-(c), we generate a sequence of $\zeta_k$ values, with $k\ge 1$, giving 
	$$4, 8, 13, 19, 26, 34, 43, \ldots  $$
	
	For $k\ge 1$ let $A_k=\{S\subseteq V(G_k)\ |\ S \ \text{is a}\ \gamma-\text{set in}\ G_k \ \text{and} \ v_{3k+1} \in S\}$,
	 $B_k=\{S\subseteq V(G_k)\ |\ S \ \text{is a}\ \gamma-\text{set in}\ G_k \ \text{and} \ v_{3k+1} \notin S\}$, and $a_k=|A_k|, b_k=|B_k|$. Then $\zeta(G_k)=a_k+b_k$. Notice that $S\in B_k \implies v_{3k}\in S$. If $S\in B_k$, there are $3$ possibilities, if $k\ge 2$:
	 \begin{enumerate}
	     \item $ v_{3k-1}, v_{3k-2} \notin S,$ in which case $S\setminus \{v_{3k}\}\in B_{k-1}.$
	      \item $ v_{3k-1} \in S.$ 
	      \item $ v_{3k-2} \in S.$ 
	 \end{enumerate}
	 
	\indent It is not possible that $\{  v_{3k-2}, v_{3k-1},  v_{3k}\subseteq S \in B_k\}$, because this inclusion would imply that $S\setminus \{v_{3k-1},  v_{3k}\} $ is a dominating set in $G_{k-1}$, which is impossible, since $|S\setminus \{v_{3k-1},  v_{3k}\}|=k-1<\gamma(G_{k-1})=k$. Therefore, for $S\in B_k$, $2$ and $3$ are disjoint possibilities.
	
	\indent In case $2$ we see that $S'=S\setminus \{v_{3k-1},  v_{3k}\}$ is a dominating set in $P_{3(k-1)}$, with $k-1=\gamma(P_{3(k-1)})$ vertices, and we know that there is only one such set, $S'=\{v_j\ | \ j\equiv 2\pmod 3, 2\le j \le 3k-4\}$.
	
	\indent In case $3$, $S\setminus \{  v_{3k}\}\in A_{k-1}$. Furthermore, if $S' \in A_{k-1}$ then $S'\cup\{v_{3k}\} \in B_k$. Similar converses apply to cases $1$ and $2$: in case $2$, if $S'$ is the unique $\gamma$-set in $P_{3(k-1)}$ then $S'\cup \{v_{3k-1},  v_{3k}\} \in B_k$, and in case 1, if $S' \in B_{k-1}$ then $S\cup \{v_{3k}\} \in B_k$. Putting these observations together, we have that $b_k=a_{k-1}+b_{k-1}+1 \ \ \ (*)$.
	
	\indent If $S\in A_k$ then $|S\setminus \{v_{3k+1}\}|=k$; since each set $T_r$, $r=0, \ldots, k-1$ must contain at least one vertex of $S\setminus \{v_{3k+1}\}$, it follows that each $T_r$ contains exactly one vertex of $S$. It follows that $v_{3r+3}\notin S$, $r=o\ldots k-1$, for if $v_{3r+3}\in S$ then, because there cannot be more than $2$ vertices between consecutive elements of $S$, it would follow that $v_{3s+3}\in S$, $0\le s \le r$. But then $v_3$ is the sole representative of $S$ in $T_0$, which contradicts the premise that $S$ is a $\gamma$-set in $G_k$. 
	
	\indent Thus, $S\in A_k$ implies $v_{3k}\notin S$, so the sole vertex in $S\cap T_{k-1}$ is either $v_{3k-1}$ or $v_{3k-2}=v_{3(k-1)+1}$. If $v_{3k-1} \in S$, then $S\setminus \{v_{3k+1}\}$ is a $\gamma$-set in $P_{3k},$ so $S=\{v_t\ | \ t=2\pmod 3, \ 2\le t\le 3k-1\} \cup \{v_{3k+1}\}$. If $v_{3k-2} \in S$, then $S\setminus \{v_{3k+1}\} \in A_{k-1}$; conversely, if $S' \in A_{k-1}$ then $v_{3k-2} \in S'$ and $S'\cup \{v_{3k-2}\}\in A_k$.
	
	\indent From the preceding paragraph we conclude that, for $k\ge 2$, $a_k=1+a_{k-1}$. Since $a_1=2$, it follows that $a_k=k+1$, $k=1,2, \ldots$
	
	\indent From $(*)$, we have that, for $k\ge 2$,
	\begin{align}
	    b_k&=a_{k-1}+b_{k-1}+1 =b_{k-1}+k+1 \nonumber \\
	    &= b_1+[(k+1)+k+\ldots + 3] \nonumber \\
	    &=2+\frac{(k+2)(k+1)}2-3 \nonumber \\
	    &=\frac{k^2+3k}2. \nonumber
	\end{align} 
	Therefore, for $k\ge 2$,
	\begin{align}
	    \zeta(G_k)&=a_k+b_k=k+1+\frac{k^2+3k}2 \nonumber \\
	    &=\frac{k^2+5k+2}2 = \frac{(n+2)(n+11)}{18}-1 \nonumber
	\end{align}
	if $n=3k+1$.

	
\end{proof}
\begin{lemma}{\label{p1}}
	If $n=3k+2$ then no $\gamma$-set of $P_n$ includes the vertex $v_3$ when $n=3k+2$. \end{lemma}
\begin{proof}
	Suppose that $k\ge 1$ and $n=3k+2$. As with any path or cycle, if $S\subseteq V(P_n)$ is dominating then $V(P_n)\setminus S$ cannot contain 3 ``consecutive'' vertices, $v_{i-1}$, $v_i$, $v_{i+1}$. Therefore, for each such $S$, each set $\ds T_r=\{v_{3r+1},v_{3r+2},v_{3r+3}\}$, $r=0,\ldots, k-1$, contains at least one vertex of $S$.
	
	\indent Now suppose that $S$ is a $\gamma$-set in $P_{n}$. Then, one of $v_{3k+1}, v_{3k+2}$ is in $S$; otherwise, $v_{3k+2}$ would not be covered. Since $|S|=3k+1$, it must be that each $T_r$ contains exactly one element of $S$. Therefore, if $v_3\in S$, then  $v_3$ is the sole representative of $S$ in $T_r$. But then $v_1$ is not covered by $S$.
\end{proof}
\begin{corollary}{\label{CC}}
If $G$ is a cycle on $n=3k$ vertices, then its dominion is
$\zeta(G)=3$.
\end{corollary}

	
\begin{proof}
	Throughout we obtain $G:=C_n$ from a path $P_{n}:=v_1-v_2-\ldots-v_{n-1}-v_{n}$ by connecting $v_1$ to $v_n$, i.e., $C_{n}:=v_1-v_2-\ldots-v_{n-1}-v_{n}-v_1$.
	\begin{claim}{\label{c5}}
		Every $\gamma$-set of $P_n$ is a $\gamma$-set of $C_n$, if $n=3k$.
	\end{claim}
	\begin{proof}
		The claim is straightforward from the fact that $\gamma(P_{3k})=\gamma(C_{3k})=k$, and the fact that any dominating set in $P_n$ will be dominating in $C_n$.
	\end{proof}

	\begin{claim}{\label{c6}}
		$\zeta(C_{3k})=\zeta(P_{3k})+2 = 3$ for all $k\ge 1$.  
	\end{claim}
	\begin{proof}
	The unique $\gamma$-set in $P_{3k}$, $S=\{v_2,v_5,\ldots,v_{3k-1}\}$, is also a $\gamma$-set of $C_{3k}$. Therefore, the $2$ rotations, say, clockwise, of this set by $1$ and $2$ places, are also $\gamma$-sets in $C_{3k}$, and are distinct subsets of $\{v_1,\ldots,v_{3k}\}$. Further rotations will just duplicate the $3$ sets already discovered. Thus, $\gamma(C_{3k})\ge 3$ and to show equality it suffices to show that any $\gamma$-set $S$ of vertices of  $C_{3k}$ must be one of these $3$:
	
	If $S$ is a $\gamma$-set in $C_{3k}$ then $|S|=k$. Then at least two consecutive vertices $u, w \in V(C_{3k}) \setminus S$. Let us rename the vertices around the cycle so that $u=v_{3k}$ and $w=v_1$. Since neither $v_1, v_{3k}$ is in $S$, yet $S$ is a dominating set in $C_{3k}$, it must be that $S$ is a dominating set in the path $P_{3k}:=v_1-\ldots-v_{3k}$. Therefore  $S=\{v_2,v_5,\ldots,v_{3k-1}\}$, which, with respect to any other naming of the consecutive vertices round $C_{3k}$, is one of those $3$ $\gamma$-sets in $C_{3k}$ already discovered.

	\end{proof}
	
	The result follows from Theorem \ref{T1} and Claim \ref{c6}.
	
\end{proof}

\begin{remark}

\end{remark}
The cases when $n=3k+1$ and $n=3k+2$ turn out to be more difficult to establish and perhaps require a novel technique. However, using Lagrange interpolation on some early recurring $\zeta$-values, we present a conjecture.

\begin{conjecture}{\label{cj1}}
	If $G$ is a cycle on $n\ge 3$ vertices, then its dominion is \\
	
	$\zeta(G)=\begin{cases} 
	\frac{n^2+5n}{6}, & \text{if} \ n \equiv 1\pmod3 \\
	n, &  \text{if} \ n \equiv 2\pmod3. \\
		\end{cases}$
\end{conjecture}


\section{Dominion of Join of Graphs }
There is no obvious comparative bound between $\gamma(G)$ and $\zeta(G)$ because, while the former determines the size of the minimum dominating set, the latter counts such sets. Yet, with the few upcoming results, we establish some relationships between these two parameters, in particular using the \textit{join} operation. Recall, the join of $G_1$ and $G_2$, commonly denoted by $G = G_1 \vee G_2$,  is the graph obtained from the disjoint union of $G_1$ and $G_2$ by adding all edges between $V(G_1)$ and $V(G_2)$.

\begin{proposition}{\label{Prop2}}
	For any  connected graph $G$ on $n$ vertices, $1\le \zeta(G)\le {n\choose \gamma(G)}$.
	
\end{proposition}
These trivial bounds follow from the fact that any $\gamma$-set is some selection from $V(G)$ and  $\zeta$ counts such sets of size $\gamma(G)$. The lower bound is tight for stars while the upper bound is tight for complete graphs.

\begin{theorem}{\label{T3}}
	Suppose $G_1$ and $G_2$ are two connected graphs with $2\le \gamma(G_1)\le \gamma(G_2)$. Then $\zeta(G_1\vee G_2)\ge |V(G_1)|\cdot|V(G_2)|$.
\end{theorem}

\begin{proof}
	Suppose $G=G_1\vee G_2$ with $2\le \gamma(G_1)\le \gamma(G_2)$. Consider any pair $\{u,v\}$ with $u\in V(G_1)$, $v\in V(G_2)$. It follows that it is a $\gamma$-set of $G$, since $u$ covers $V(G_2)$ and $v$ covers $V(G_1)$, giving the bound. 
\end{proof}
The previous argument also establishes that $1\le \gamma(G_1\vee G_2)\le 2$, given any two graphs $G_1$, $G_2$.

\begin{theorem}{\label{T4}}
	Suppose $G_1$ and $G_2$ are two connected graphs with $\gamma(G_1)\le \gamma(G_2)$. Then  \\
	$\gamma(G_1\vee G_2)=\begin{cases} 
	1 &  \text{if} \  \gamma(G_1)=1  \\
	2 & \text{if}  \ \gamma(G_1)\ge 2.
	\end{cases}$
	\\
	Moreover,
	\\
	$\zeta(G_1\vee G_2)=\begin{cases} 
	\zeta(G_1)+\zeta(G_2)+ \begin{cases} 
	0 & \text{if} \  \gamma(G_1)=1=\gamma(G_2) \\
	|V(G_1)|\cdot|V(G_2)| , &  \text{if} \  \gamma(G_1)=2=\gamma(G_2)  
	\end{cases} 
	\\
	\\
	\zeta(G_1) \hspace{.55in} +\begin{cases} 
	0 & \text{if} \  \gamma(G_1)=1<\gamma(G_2) \\
	|V(G_1)|\cdot|V(G_2)| , &  \text{if} \  \gamma(G_1)=2<\gamma(G_2)  
	
\end{cases} 
\\
\\
|V(G_1)|\cdot|V(G_2)|,   \hspace{1.25in} \ \ \text{if} \  2<\gamma(G_1)\le \gamma(G_2).	
\end{cases}$
\end{theorem}


\begin{proof}
The first part of the statement regarding the values of $\gamma(G_1\vee G_2)$ is trivial.

So, we proceed to establish the results in the second part of the statement. Throughout, we denote by $S_1$ and $S_2$ any $\gamma$-sets of $G_1$ and $G_2$, respectively. 
\newline
(a)
\newline
Suppose $\gamma(G_1)=1=\gamma(G_2)$. Given the join operation, any vertex $v_1\in S_1$ also covers any other vertex $v_2\in V(G_2)$. Likewise, any vertex $v_2\in S_2$ also covers any other vertex $v_1\in V(G_1)$. Since $S_1$ and $S_2$ are any $\gamma$-sets of $G_1$ and $G_2$, respectively, the result follows.
\newline
(b)
\newline
Suppose $\gamma(G_1)=2=\gamma(G_2)$. 
It is clear that each $\gamma$-set $S_i$, $i=1,2,$ is also a $\gamma$- set of $G_1\vee G_2$, following from $(a)$. The result follows the argument presented in the proof of Theorem \ref{T3}.
\newline
(c)
\newline
Suppose $\gamma(G_1)<\gamma(G_2)$ with $\gamma(G_1) \in \{1,2\}$. It follows that $\gamma(G_1\vee G_2)=\gamma(G_1)$. Further, if $\gamma(G_1)=1$, then $S_1$ is also a $\gamma$-set of $G_1\vee G_2$ and $S_2$ is not, since its size is greater than 1. Similarly, if $\gamma(G_1)=2$, $S_1$ is a $\gamma$-set of $G_1\vee G_2$ but not $S_2$. The result follows from the argument presented in the proof of Theorem \ref{T3}.
\newline
(d)
\newline
Finally, suppose $2<\gamma(G_1)\le \gamma(G_2)$. 
It must be that $\gamma(G_1\vee G_2)=2$. In which case, neither $S_1$ nor $S_2$ is a $\gamma$-set of $G_1\vee G_2$. The result follows from the argument presented in the proof of Theorem \ref{T3}.

\end{proof}
\begin{corollary}
For any connected graph $G$ with $\gamma(G)=1$, $\zeta(\underbrace{G\vee \ldots \vee G}_r)=r\zeta(G)$, $r\ge 1$.
\end{corollary}
\begin{proof}
From Part (a) of Theorem \ref{T4} (when $\gamma(G_1)=1=\gamma(G_2)$), the result follows by induction on $r\ge 1$.
\end{proof}
We note that in the case when $G=K_1$, it is clear that $\underbrace{G\vee \ldots \vee G}_r:=K_r$ and $\zeta(K_r)=r=r\zeta(K_1)$. Here, we present a more general result for the special case when $G=\overline{K_{m}}$, the complement of a complete graph on $m\ge 1$ vertices. Note that $\overline{K_{m_1}} \vee \ldots \vee \overline{K_{m_k}}$ is a complete $k$-partite graph often written as $K(m_1,m_2,\ldots, m_k)$, with $m_i\ge 1$ and $k\ge 2$.

\begin{theorem}
Suppose  $G=K(m_1,m_2,\ldots, m_k)$ is a complete $k$-partite graph with $m_1\le m_2 \le \ldots \le m_k$. For some $r\ge 1$, we have

$\zeta(G)=\begin{cases} 
r &  \text{if} \  m_i=1,\  i=1, \ldots, r  \ and \  m_i>1, \ r<i\le k,\\
\ds \sum_{\substack {1\le i,j \le r \\ i\ne j}}m_im_j \ + &\begin{cases}
r & \text{if} \ m_i=2, \ i=1, \ldots, r  \ and \  m_i>r, \ r<i\le k, \\
0 & \text{if} \ m_i>2 \ i=1, \ldots, r=k.
\end{cases}
\end{cases}$


\end{theorem}
\begin{proof}
Let $V_i$ denote each disjoint partite (vertex) set of $G$, with $|V_i|=m_i$, $i=1,\ldots, k$. If, for some $j=1,\ldots, r$, $m_j=1$ and $m_i>1$ for $i>r$, then it is clear that $\gamma(G)=1$ and the $\gamma$-sets in $G$ are the singleton subsets of $V_1\cup\ldots \cup V_r$. It follows that $\zeta(G)=r$. 

Assume $m_i>1$, for $i=1,\ldots, k$. It is clear that $\gamma(G)=2$. Any vertex $v\in V_i$ covers all vertices $v'\in V(G)\setminus V_i$ except those in $V_i$. Thus, for each $v\in V_i$ and $v'\in V_j$, $i\ne j$, $\{v,v'\}$ is a $\gamma$-set of $G$. It follows that there are $\ds \sum_{i\ne j}^{{k}}m_im_j$ total such sets. Further, if  $|V_i|=2$ for $i=1,\ldots, r$, then each such sets, $V_i$, is an additional $\gamma$-set, giving the result.
\end{proof}

 {\bf Acknowledgment.}
The authors are indebted to the reviewer, Pete Johnson Jr., for his careful reading of the article and for his many corrections and suggestions; they have greatly enhanced the quality of this research.

\end{document}